\documentclass[12pt]{amsart}
\usepackage{amsmath}
\usepackage{amssymb,latexsym}
\usepackage{enumerate}
\usepackage{amsbsy}
\usepackage{mathrsfs}
\usepackage[all]{xypic}

\makeatletter
\@namedef{subjclassname@2010}{%
  \textup{2010} Mathematics Subject Classification}
\makeatother

\newtheorem{theorem}{Theorem}[section]
\newtheorem{prop}[theorem]{Proposition}
\newtheorem{lemma}[theorem]{Lemma}
\newtheorem{cor}[theorem]{Corollary}
\newtheorem{definition}[theorem]{Definition}

\newenvironment{remark}{\vspace{0.25cm} \noindent \textbf{Remark}.}{\vspace{0.25cm}}

\numberwithin{equation}{section}

\newcommand{\ra}{\rightarrow}

\newcommand{\bb}[1]{\mathbb{#1}}

\newcommand{\Z}{\bb{Z}}
\newcommand{\Q}{\bb{Q}}
\newcommand{\R}{\bb{R}}

\newcommand{\p}{\mathfrak{p}}
\newcommand{\pp}{\mathfrak{P}}
\newcommand{\q}{\mathfrak{q}}
\newcommand{\qq}{\mathfrak{Q}}
\newcommand{\ten}{\otimes}
\newcommand{\teno}[1]{\ten_{#1}}

\newcommand{\gal}[1]{\mathrm{Gal}(#1)}

\newcommand{\br}[1]{\bar{#1}} 

\newcommand{\st}{^\times}
\newcommand{\ie}{i.e. }

\newcommand{\can}{\simeq}
\newcommand{\Hom}{\mathrm{Hom}}

\newcommand{\setm}{\! \smallsetminus \!}

\newcommand{\Spec}{\mathrm{Spec}}

\newcommand{\oh}{\mathcal{O}}
\newcommand{\class}{\mathrm{Cl}}

\newcommand{\spc}{\mbox{ }}

\newcommand{\con}{\subseteq}

\newcommand{\sat}{\spc | \spc}

\newcommand{\chdot}{_{\textrm{\tiny{$\bullet$}}}}

\newcommand{\Ker}{\mathrm{Ker}}
\newcommand{\im}{\mathrm{Im}}

\newcommand{\map}{\mathrm{Map}}

\newcommand{\ab}{^{\mathrm{ab}}}

\newcommand{\frob}{\varphi}

\newcommand{\Ext}{\mathrm{Ext}}

\newcommand{\aug}[1]{\Delta #1}
\newcommand{\raug}[2]{\Delta_{#1} #2}
\newcommand{\augm}{\mathrm{aug}}

\newcommand{\ur}{^{\mathrm{ur}}}
\newcommand{\rc}[1]{\overline{#1}}

\newcommand{\commut}[2]{[#1,#2]}

\newcommand{\minigap}{\vspace{0.1cm}}

\newcommand{\frakc}{\mathfrak{c}}

\newcommand{\calF}{\mathcal{F}}
\newcommand{\inorm}{\mathrm{Nm}}
\newcommand{\calC}{\mathcal{C}}
\newcommand{\calV}{\mathcal{V}}

\newcommand{\Left}{_\mathrm{left}}
\newcommand{\calH}{\mathcal{H}}

\begin{document}


\baselineskip=17pt



\title[Connecting homomorphisms associated to Tate sequences]{Connecting homomorphisms associated to Tate sequences}

\author[P. R. Buckingham]{Paul Richard Buckingham}
\address{Department of Mathematical and Statistical Sciences \\
632 CAB \\
University of Alberta \\
Edmonton AB T6G 2G1 \\
Canada}
\email{p.r.buckingham@ualberta.ca}

\date{}

\begin{abstract}
Tate sequences are an important tool for tackling problems related to the (ill-understood) Galois structure of groups of $S$-units. The relatively recent Tate sequence ``for small $S$'' of Ritter and Weiss allows one to use the sequence without assuming the vanishing of the $S$-class-group, a significant advance in the theory. Associated to Ritter and Weiss's version of the sequence are connecting homomorphisms in Tate cohomology, involving the $S$-class-group, that do not exist in the earlier theory. In the present article, we give explicit descriptions of certain of these connecting homomorphisms under some assumptions on the set $S$.
\end{abstract}



\maketitle

\section{Introduction} \label{intro}

In \cite{tate:tori,tate:stark}, Tate constructs, for a Galois extension $L/K$ of number fields with Galois group $G$, a \emph{Tate sequence} $0 \ra \oh_{L,S}\st \ra A \ra B \ra X \ra 0$ of $\Z[G]$-modules whose extension class in $\Ext_{\Z[G]}^2(X,\oh_{L,S}\st)$ is the so-called \emph{Tate canonical class}, where $\oh_{L,S}\st$ is the group of $S$-units in $L$ and $X$ is a finitely generated, $\Z$-torsion-free module which will be defined in Section \ref{notation}. This is done under the assumptions that $S$ contains the ramified places and that the $S$-class-group of $L$ is trivial.

Tate's construction shows that $A$ and $B$ may in fact be chosen finitely generated and cohomologically trivial. This has the consequence that the Galois cohomology of $\oh_{L,S}\st$ can be identified with that of $X$, after a dimension shift of $2$. This is reminiscent of the Artin--Tate formulation of class field theory \cite{at:cft} in which the cohomology of the idele class-group is identified with the cohomology of $\Z$ after a dimension shift of $2$, with a similar statement in the local case. In fact, this is no coincidence, since it is via this interpretation of class field theory that the Tate sequence is constructed.

The sequence is a primary feature in numerous applications to multiplicative Galois module structure. Tate uses it in \cite[Ch. II]{tate:stark} to prove Stark's Conjecture for rational characters, and Chinburg employs it in the construction of his third $\Omega$-invariant, which is central to his root number conjecture -- see \cite{chinburg:galstruct,chinburg:exactseq}. In fact, the Tate sequence also appears in the definition of the lifted $\Omega$-invariant in the Lifted Root Number Conjecture of Gruenberg--Ritter--Weiss \cite{grw:lrnc}. Further, a variant of the sequence is used to construct the equivariant Tamagawa number in the Equivariant Tamagawa Number Conjecture (ETNC) for the motive $h^0(\Spec(L))(0)$, as in Burns--Flach \cite{bf:tamagawa2} for example.

A significant step forward in the theory of Tate sequences was Ritter and Weiss's Tate sequence ``for small $S$'', which allows the sequence to be constructed for an arbitrary set of places $S$. More precisely, they found in \cite{rw:tateforglobal} a sequence
\begin{equation} \label{tate sequence}
0 \ra \oh_{L,S}\st \ra A \ra B \ra \nabla \ra 0
\end{equation}
for a Galois extension $L/K$ of number fields with Galois group $G$, with no restriction on the set $S$ other than that it contains the infinite places. Whereas in the preceding theory of Tate sequences, the module in place of $\nabla$ was the $\Z$-torsion-free $\Z[G]$-module $X$ whose Galois structure is described in terms of very basic arithmetic information, $\nabla$ itself fits into an exact sequence
\begin{equation} \label{nabla extension}
0 \ra \class_S(L) \ra \nabla \ra X \ra 0 .
\end{equation}
An early use of this refined Tate sequence ``for small $S$'' was in the proof by Ritter and Weiss \cite{rw:cohomology} of the Strong Stark Conjecture for abelian extensions of $\Q$. More recent applications have been the consideration of the minus part of the ETNC for tame extensions, as in Nickel \cite{nickel:tame1}, and the study of Fitting ideals of (duals of) minus parts of class-groups, as in Greither \cite{greither:fitting}.

Despite the sequence's growing use in multiplicative Galois module structure, one aspect of it that has not yet been determined is the precise effect of the connecting homomorphisms that are naturally associated to the sequence. Our aim is to make certain such maps explicit, under some assumptions on the set $S$, namely that $S$ will contain the infinite and ramified places and at least one place with full decomposition group. We expect to be able to remove this last assumption in the future. However, we emphasize that the $S$-class-group will not be assumed trivial. The maps we make explicit are the connecting homomorphism $H^{-2}(G,X) \ra H^{-1}(G,\class_S(L))$ associated to (\ref{nabla extension}), and the map $H^{-1}(G,\class_S(L)) \ra H^1(G,\oh_{L,S}\st)$ that results from $H^{-1}(G,\class_S(L)) \ra H^{-1}(G,\nabla)$ together with the appropriate connecting homomorphisms (\ie starting in dimension $-1$ then $0$) obtained by splitting (\ref{tate sequence}) into two short exact sequences.

After building up preliminaries in Sections \ref{notation}, \ref{sec: group coh}, \ref{modules} and \ref{sec: image of snake}, we compute the maps in Sections \ref{X Cl map} and \ref{Cl U map}, and give some corollaries in Section \ref{cors}.

\section{Notation, assumptions and conventions} \label{notation}

We will assume throughout that $L/K$ is Galois with Galois group $G$. $S$ will denote a finite set of places of $K$ containing the infinite and ramified places. By a minor abuse of notation, $\oh_{L,S}$ will denote the $S_L$-integers in $L$, where $S_L$ consists of the places of $L$ above those in $S$. Thus $\oh_{L,S}\st$ is the group of $S_L$-units in $L$. The $S$-class-group of $K$ will be written $\class_S(K)$, and the $S_L$-class-group of $L$ will be just $\class_S(L)$, the redundant subscript $L$ again being dropped.

By $L_S$, we mean the Hilbert $S_L$-class field of $L$, that is, the maximal unramified abelian extension of $L$ in which all places in $S_L$ split completely. $\gal{L_S/L}$ is isomorphic, via the Artin map, to $\class_S(L)$.

We will denote the element $\sum_{\sigma \in G} \sigma \in \Z[G]$ by $N$. Multiplication by $N$ induces an endomorphism of $\class_S(L)$, but we stress that this is different from the map $\inorm : \class_S(L) \ra \class_S(K)$, induced by the norm of ideals, that is featured in Section \ref{cors}.

All number fields will lie in a fixed algebraic closure $\br{\Q}$ of $\Q$. If $\p$ is a place of $K$, we will fix once and for all a place of $\br{\Q}$ above $\p$. Then given any number field $F$ containing $K$, $\p(F)$ will denote the place of $F$ below the chosen place of $\br{\Q}$ above $\p$. For shorthand, we will denote the completion $F_{\p(F)}$ by $F_\p$, and if $F/K$ is Galois, the decomposition group $\gal{F/K}_{\p(F)}$ will be denoted simply $\gal{F/K}_\p$. Note that if $\p \in S$, then since $\p(L)$ splits completely in $L_S$, restriction $\gal{L_S/K}_\p \ra G_\p$ defines an isomorphism, and we denote the composition
\[ G_\p \stackrel{\can}{\ra} \gal{L_S/K}_\p \ra \gal{L_S/K} \]
by $\iota_\p$.

For a finite place $\p$ of $K$, let $\tilde{\frob}_\p$ denote a lift in $\gal{L_\p\ur/K_\p}$ of the Frobenius of $K_\p\ur/K_\p$, and $\br{\frob}_\p$ its image in $\gal{L_S/K}$. Further, if $\p$ is unramified in $L/K$, $\frob_\p$ will be the associated Frobenius element in $G_\p$.

An important object appearing throughout the article is the Galois module $X$ defined in Definition \ref{def X}:

\begin{definition} \label{def X}
Let $Y$ be the free abelian group on $S_L$, and $X$ the kernel of the augmentation map $Y \ra \Z$ sending every place $\pp \in S_L$ to $1$.
\end{definition}

To give some context, we remark that $X$ appears in the Dirichlet regulator map, \ie the isomorphism
\begin{eqnarray*}
\R \teno{\Z} \oh_{L,S}\st &\ra& \R \teno{\Z} X \\
1 \ten u &\mapsto& \sum_{w \in S_L} \log \|u\|_w w ,
\end{eqnarray*}
where the absolute values $\|\cdot\|_w$ are normalized in the particular canonical way which makes the product formula hold, as in \cite[Ch. III, Section 1]{neukirch:alg}.

For any group $G$, $\aug{G}$ will denote its augmentation ideal, that is, the kernel of the augmentation map $\Z[G] \ra \Z$ which sends each group element to $1$.

\subsection{Key assumptions}

In Section \ref{modules} and Section \ref{sec: desc s}, we assume that $L/K$ is Galois and $S$ contains the ramified places. In Section \ref{sec: ext of nabla} and Sections \ref{X Cl map}, \ref{Cl U map} and \ref{cors}, we further assume that there is a place $\p_0 \in S$ such that $G_{\p_0} = G$. We observe that this last assumption forces $G$ to be solvable, since then $G$ is the Galois group of a Galois extension of local fields.

\section{Group cohomology} \label{sec: group coh}

Cohomology will be Tate cohomology throughout. We will use the following model for Tate cohomology groups in negative degrees, which we shall think of as homology groups, that is $H^i(G,A) = H_{-i-1}(G,A)$ for $i < -1$: For $n \geq 0$, let $P_n$ be the free \emph{right} $\Z[G]$-module on $G^n$, and for $n \geq 1$ define a boundary map $d_n : P_n \ra P_{n-1}$ by sending $[\sigma_1,\ldots,\sigma_n]$ to
\[ [\sigma_1,\ldots,\sigma_{n-1}] \cdot \sigma_n + \sum_{i=1}^{n-1} (-1)^{n-i} [\sigma_1,\ldots,\sigma_i \sigma_{i+1},\ldots,\sigma_n] + (-1)^n [\sigma_2,\ldots,\sigma_n] .\]
Then $P\chdot \stackrel{\augm}{\ra} \Z \ra 0 $ is a free resolution of $\Z$ as a \emph{right} $\Z[G]$-module, and for a \emph{left} $\Z[G]$-module $A$, we view $H_i(G,A)$ as the $i$th homology group of the chain complex $P\chdot \teno{\Z[G]} A$ when $i > 0$. The Tate cohomology group $H^{-1}(G,A)$ will be taken to mean
\begin{equation} \label{def H-1}
\frac{\{a \in A \sat \sum_{\sigma \in G} \sigma a = 0\}}{\Z \{(\sigma - 1)a \sat \sigma \in G, a \in A\}} .
\end{equation}
Note that the denominator in (\ref{def H-1}) is indeed a $\Z[G]$-submodule, \ie is closed under the action of $G$.

\begin{lemma} \label{Hab iso}
Let $G$ be a finite group and $H$ a subgroup. There is a well-defined group isomorphism
\[ H\ab \ra H^{-2}(G,\Z[G] \teno{\Z[H]} \Z) \]
sending $\sigma \commut{H}{H}$ to $\rc{[\sigma] \ten 1 \ten 1}$.
\end{lemma}

\begin{proof}
That $H\ab$ is isomorphic to $H^{-2}(G,\Z[G] \teno{\Z[H]} \Z)$ is simply Shapiro's Lemma -- as in \cite[Lemma 6.3.2]{weibel:ha} for example -- together with the fact that $H^{-2}(H,\Z) \can H\ab$. The explicit description of the isomorphism is left to the reader.
\end{proof}

\subsection{Extension classes}

Suppose that $A$ and $C$ are $\Z[G]$-modules and that $C$ is $\Z$-free.

\begin{lemma} \label{ext and h1}
There is a canonical isomorphism
\[ \Ext_{\Z[G]}^1(C,A) \can H^1(G,\Hom_\Z(C,A)) .\]
\end{lemma}

\begin{proof}
Since this is well known, we only sketch the proof. Supposing we have an exact sequence $0 \ra A \ra B \ra C \ra 0$ of $\Z[G]$-modules, choose a section $s : C \ra B$ of the $\Z$-module homomorphism $B \ra C$, so that $s \in \Hom_\Z(C,B)$ maps to $\mathbf{1}_C$ in $\Hom_\Z(C,C)$. This is possible by the assumption on $C$. Then for each $\sigma \in G$, $\sigma s - s$ is the image of a unique $f_\sigma \in \Hom_\Z(C,A)$. The map $\sigma \mapsto f_\sigma$ is a $1$-cocycle $G \ra \Hom_\Z(C,A)$.

Conversely, given a $1$-cocycle $f : G \ra \Hom_\Z(C,A)$, we can endow the direct sum $B = A \oplus C$ of $\Z$-modules with a $G$-action by
\begin{equation} \label{G-action for h1}
\sigma(a,c) = (\sigma a + f(\sigma)(\sigma c),\sigma c) .
\end{equation}
One checks that these constructions pass to mutually inverse maps between $\Ext_{\Z[G]}^1(C,A)$ and $H^1(G,\Hom_\Z(C,A))$.
\end{proof}

\begin{lemma} \label{conn as cup}
Suppose $0 \ra A \ra B \ra C \ra 0$ is an exact sequence of $\Z[G]$-modules such that $C$ is $\Z$-free. Viewing its extension class $\xi$ as an element of $H^1(G,\Hom_\Z(C,A))$, the connecting homomorphism $H^i(G,C) \ra H^{i+1}(G,A)$ is given by following cup-product with $\xi$ by the evaluation map
\[ H^{i+1}(G,C \teno{\Z} \Hom_\Z(C,A)) \ra H^{i+1}(G,A) .\]
\end{lemma}

\begin{proof}
See \cite[Ch. XII, Prop. 6.1]{ce:homalg}.
\end{proof}

By Lemma \ref{ext and h1}, $\Ext_{\Z[G]}^1(\aug{G},A) \can H^1(G,\Hom_\Z(\aug{G},A))$ for each $\Z[G]$-module $A$. On the other hand, the exact sequence
\[ 0 \ra A \ra \Hom_\Z(\Z[G],A) \ra \Hom_\Z(\aug{G},A) \ra 0 \]
together with the cohomological triviality of $\Hom_\Z(\Z[G],A)$ gives an isomorphism $H^1(G,\Hom_\Z(\aug{G},A)) \ra H^2(G,A)$. Thus we have:

\begin{lemma} \label{h2 as ext}
For any $\Z[G]$-module $A$, $\Ext_{\Z[G]}^1(\aug{G},A) \can H^2(G,A)$.
\end{lemma}

\section{The modules $W_{S'}$, $R_{S'}$ and $B_{S'}$} \label{modules}

\subsection{The modules $W_{S'}$ and $R_{S'}$}

We let $S'$ denote a finite set of places of $K$ containing $S$, and further satisfying:

\minigap

\begin{tabular}{rp{10cm}}
(i) & $\bigcup_{\p \in S'} G_\p = G$, \\
and (ii) & $\class_{S'}(L) = 0$.
\end{tabular}

\minigap

\noindent Such an $S'$ can always be chosen, by the Chebotarev Density Theorem. Following \cite[Section 1]{rw:tateforglobal}, we define
\[ W_{S'} = \bigoplus_{\p \in S} \raug{\p}{G} \oplus \bigoplus_{\q \in S' \setm S} \Z[G] ,\]
where $\raug{\p}{G}$ is the left ideal in $\Z[G]$ generated by $\aug{G_\p}$. We observe that this description of $W_{S'}$ relies on $S$ containing the ramified primes. For a treatment of the general case, see \cite{rw:tateforglobal} itself.

As in \cite[Section 4]{rw:tateforglobal}, we now define $R_{S'}$ to be the kernel of the map $W_{S'} \ra \aug{G}$ that is inclusion on $\raug{\p}{G}$ and left multiplication by $\frob_\q - 1$ on the copy of $\Z[G]$ corresponding to $\q \in S' \setm S$.

\subsection{The module $B_{S'}$} \label{def b}

After \cite[Section 4]{rw:tateforglobal}, we let $B_{S'}$ be the kernel of the map
\[ \bigoplus_{\p \in S'} \Z[G] \ra \Z[G] \]
that is the identity on the copy of $\Z[G]$ corresponding to $\p \in S$, and multiplication by $\frob_\q - 1$ on the copy of $\Z[G]$ corresponding to $\q \in S' \setm S$. We note that $B_{S'}$ is projective, and therefore cohomologically trivial. To provide more context, we note that $B_{S'}$ is the module $B$ appearing in (\ref{tate sequence}). The map $R_{S'} \ra B_{S'}$ in (\ref{rbx}) is induced by the inclusion map
\[ W_{S'} \ra \bigoplus_{\p \in S'} \Z[G] .\]
The map $B_{S'} \ra X$ is induced by the map
\[ \bigoplus_{\p \in S'} \Z[G] \ra Y \]
that is zero on $\Z[G]$ corresponding to $\q \in S' \setm S$, and sends $\alpha \in \Z[G]$ corresponding to $\p \in S$ to $\alpha \p(L)$.

\subsection{Class field theory and diagrams}

Denote the idele class-group of $L$ by $C_L$. We let
\[ 0 \ra C_L \ra \calV \ra \aug{G} \ra 0 \]
be the extension corresponding under the isomorphism of Lemma \ref{h2 as ext} to the global fundamental class in $H^2(G,C_L)$. In fact, for concreteness, we will take the following description of $\calV$: Suppose the element of $H^1(G,\Hom_\Z(\aug{G},C_L))$ corresponding to the global fundamental class is represented by the $1$-cocycle $f$. Then we view $\calV$ as $C_L \oplus \aug{G}$ (direct sum as $\Z$-modules) with $G$-action given as in (\ref{G-action for h1}).

Ritter and Weiss also define local versions $V_\p$ of $\calV$ for each $\p \in S'$. For $\p \in S$, $V_\p$ is simply the analogous construction to $\calV$ with $C_L$ replaced by $L_\p\st$ and the global fundamental class replaced by the local one. For $\p \in S' \setm S$, the definition is more subtle, but still uses local class field-theoretic data. We will only need the definitions in a very simple situation, which we will turn to in Section \ref{sec: delta2}. The reader wishing to see the complete definition may refer to \cite[Sections 1,3]{rw:tateforglobal}. Following Ritter and Weiss, we set
\begin{equation} \label{vs' def}
V_{S'} = \left(\bigoplus_{\p \in S'} \Z[G] \teno{\Z[G_\p]} V_\p\right) \oplus \left(\bigoplus_{\pp \not\in S_L'} U_\pp\right) ,
\end{equation}
where $U_\pp$ is the group of units in $L_\pp$. Note that the first direct sum runs through primes $\p$ of $K$ (in $S'$), whereas the second direct sum runs through primes $\pp$ of $L$ (not above $S'$).

The important diagrams involving the modules $W_{S'}$, $R_{S'}$ and $B_{S'}$ are 
\begin{equation} \label{big diagram}
\xymatrix@R=10pt{
& & 0 \ar[dd] & 0 \ar[dd] & 0 \ar[dd] & & \\
& & & & & & \\
& 0 \ar[r] & \oh_{L,S}\st \ar[r] \ar[dd] & A \ar[r] \ar[dd] & R_{S'} \ar[dd] \ar `r[rrd]^s `[ddd] `[llllddddd] `[dddddd] [lldddddd] & & \\
& & & & & & \\
& 0 \ar[r] & J_S \ar[r] \ar[dd] & V_{S'} \ar[r] \ar[dd] & W_{S'} \ar[r] \ar[dd] & 0 & \\
& & & & & & \\
& 0 \ar[r] & C_L \ar[r] \ar[dd] & \mathcal{V} \ar[r] \ar[dd] & \aug{G} \ar[r] \ar[dd] & 0 & \\
& & & & & & \\
& & \class_S(L) \ar[r] \ar[dd] & 0 & 0 & & \\
& & & & & & \\
& & 0 & & & &
}
\end{equation}
and
\begin{equation} \label{RBX diagram}
\xymatrix{
0 \ar[r] & R_{S'} \ar[r] \ar[d]^s & B_{S'} \ar[r] \ar[d]^t & X \ar[r] \ar[d]^1 & 0 \\
0 \ar[r] & \class_S(L) \ar[r] & \nabla \ar[r] & X \ar[r] & 0 .
}
\end{equation}
In (\ref{big diagram}), the module $A$ is defined to be the kernel of $V_{S'} \ra \mathcal{V}$, and (\ref{big diagram}) is simply the snake diagram arising from the middle two rows. The surjectivity of $V_{S'} \ra \calV$ in (\ref{big diagram}) and the origin of (\ref{RBX diagram}) are treated in \cite[Section 4]{rw:tateforglobal}.

$R_{S'}$ fits into an exact sequence
\begin{equation} \label{pretate}
0 \ra \oh_{L,S}\st \ra A \ra R_{S'} \stackrel{s}{\ra} \class_S(L) \ra 0 ,
\end{equation}
where $A$ is the (cohomologically trivial) $\Z[G]$-module appearing in the Tate sequence (\ref{tate sequence}). Ritter and Weiss term the map $s$ the `snake map', because it is the snake map of diagram (\ref{big diagram}) -- see the discussion following \cite[Theorem 1]{rw:tateforglobal}. The significance of the snake map is as follows: Applying $\Hom_\Z(-,\class_S(L))$ to the short exact sequence
\begin{equation} \label{rbx}
0 \ra R_{S'} \ra B_{S'} \ra X \ra 0
\end{equation}
occurring in (\ref{RBX diagram}), and remembering that $X$ is $\Z$-free, we obtain the exact sequence
\[ 0 \ra \Hom_\Z(X,\class_S(L)) \ra \Hom_\Z(B_{S'},\class_S(L)) \ra \Hom_\Z(R_{S'},\class_S(L)) \ra 0 .\]
Identifying $\Ext_{\Z[G]}^1(X,\class_S(L))$ with $H^1(G,\Hom_\Z(X,\class_S(L)))$, the extension class of (\ref{nabla extension}) is the image of the class of $-s$ under the connecting homomorphism
\[ H^0(G,\Hom_\Z(R_{S'},\class_S(L))) \ra H^1(G,\Hom_\Z(X,\class_S(L))) .\]
Since (\ref{nabla extension}) is central to our purpose, getting a strong hold on $s$ will be one of our main goals, and will be carried out in Section \ref{sec: image of snake}.

\section{The snake map} \label{sec: image of snake}

\subsection{Explicit description of $s$} \label{sec: desc s}

In this section, $L/K$ is a Galois extension of number fields and $S$ a finite set of places of $K$ containing the infinite and ramified ones. We let $G^S = \gal{L_S/K}$.

We now come to the description of the snake map $s : R_{S'} \ra \class_S(L)$ given in \cite[Section 5]{rw:tateforglobal}. It is defined first of all as a prime-by-prime map $W_{S'} \ra \calH$ where
\[ \calH = \frac{\aug{G^S}}{\Ker(\aug{G^S} \ra \aug{G}) \aug{G^S}} .\]
(We note that an element $\sigma$ of $G$ acts on $\calH$ by left multiplication by any preimage of $\sigma$ in $G^S$.)

For $\p \in S$, the map $\raug{\p}{G} \ra \calH$ sends $\beta \alpha$ to $\beta(\rc{\iota_\p(\alpha)})$ for $\alpha \in \aug{G_\p}$ and $\beta \in \Z[G]$, where the bar denotes class in $\calH$.

Recall the element $\br{\frob}_\q \in G^S = \gal{L_S/K}$ defined in Section \ref{notation}. For $\q \in S' \setm S$, the image of an element $\beta$ inside the copy of $\Z[G]$ in $W_{S'}$ corresponding to $\q$ is mapped to $\beta(\rc{\br{\frob}_\q})$ in $\calH$.

There is an embedding $\gal{L_S/L} \ra \calH$ which sends $\sigma$ to the class of $\sigma - 1$, and the restriction of the map $W_{S'} \ra \calH$ to $R_{S'}$ has its image in $\im(\gal{L_S/L} \ra \calH)$. Identifying $\gal{L_S/L}$ with $\class_S(L)$, we have thus described the snake map $s : R_{S'} \ra \class_S(L)$.

\subsection{The extension class of $\nabla$} \label{sec: ext of nabla}

We now assume that there exists a finite place $\p_0 \in S$ such that $G_{\p_0} = G$. For each $\p \in S$, fix a set $D_\p$ of representatives for $(G/G_\p)\Left$ containing $1$, and given $\sigma \in G$ let $\rho_\p(\sigma)$ be the chosen representative of the coset $\sigma G_\p$. It is possible to choose $S'$ as in Section \ref{modules} with the further condition that every place $\q \in S' \setm S$ splits completely in $L/K$. For this, we are already using that $\bigcup_{\p \in S} G_\p = G$. We are assuming more, of course: $G_{\p_0} = G$. With $S'$ chosen in this way, we have 
\begin{equation} \label{nice R}
R_{S'} = \Ker\left(\bigoplus_{\p \in S} \raug{\p}{G} \ra \aug{G}\right) \oplus \bigoplus_{\q \in S' \setm S} \Z[G] .
\end{equation}

\begin{definition}
Given $\p \in S \setm \{\p_0\}$, $\sigma \in G$ and $\tau \in D_\p$, let $r_{\sigma,\tau}^{(\p)}$ be the element of $R_{S'} = \Ker(\bigoplus_{\p \in S} \raug{\p}{G} \ra \aug{G}) \oplus \bigoplus_{\q \in S' \setm S} \Z[G]$ which has $\sigma \rho_\p(\sigma^{-1} \tau) - \tau$ in the $\p$-component, $\tau - \sigma \rho_\p(\sigma^{-1} \tau)$ in the $\p_0$-component, and zero elsewhere.
\end{definition}

\begin{lemma} \label{image of snake}
Define $g : G \ra \Hom_\Z(X,\class_S(L))$ by sending $\sigma \in G$ to the map
\[ \tau \p(L) - \p_0(L) \mapsto s(r_{\sigma,\tau}^{(\p)}) \]
for $\tau \in D_\p$, $\p \in S \setm \{\p_0\}$. Then the image of $-\rc{s} \in H^0(G,\Hom_\Z(R_{S'},\class_S(L)))$ in $H^1(G,\Hom_\Z(X,\class_S(L)))$ is $\rc{g}$.
\end{lemma}

\begin{proof}
We follow the diagram
\begin{equation} \label{conn calc}
\xymatrix{
& \map(G,\Hom_\Z(X,\class_S(L))) \ar[d] \\
\Hom_\Z(B_{S'},\class_S(L)) \ar[r] \ar[d] & \map(G,\Hom_\Z(B_{S'},\class_S(L))) \\
\Hom_\Z(R_{S'},\class_S(L)) & 
}
\end{equation}
from bottom-left to top-right. Note that
\[ B_{S'} = \Ker\left(\bigoplus_{\p \in S} \Z[G] \ra \Z[G] \right) \oplus \bigoplus_{\q \in S' \setm S} \Z[G] .\]

We first look for a $\Z$-splitting of $0 \ra R_{S'} \ra B_{S'} \ra X \ra 0$. We use the $\Z$-basis
\[ \bigcup_{\p \in S \setm \{\p_0\}} \{\tau \p(L) - \p_0(L) \sat \tau \in D_\p\} \]
of $X$. For $\p \in S$, let $e_\p$ be the element of $\bigoplus_{\p \in S} \Z[G]$ which has $1$ in the $\p$-component and zero everywhere else. Then a lift of $\tau \p(L) - \p_0(L)$ to $B_{S'}$ is $\tau e_\p - \tau e_{\p_0}$. Therefore if $Z$ is the $\Z$-span of these lifts,
\[ B_{S'} = R_{S'} \oplus Z .\]

Define $\mu : B_{S'} \ra \class_S(L)$ by $\mu|_{R_{S'}} = -s$ and $\mu|_Z = 0$. Let $f$ be the image of $\mu$ under the horizontal map in (\ref{conn calc}), \ie if $\sigma \in G$, $f(\sigma) = \sigma \mu - \mu$. Now take $\tau \in D_\p$ for some $\p \in S$.
\[ (\sigma \mu)(\tau e_\p - \tau e_{\p_0}) = \sigma \mu(\sigma^{-1} \tau e_\p - \sigma^{-1} \tau e_{\p_0}) ,\]
and one sees that $\sigma^{-1} \tau e_\p - \sigma^{-1} \tau e_{\p_0}$ decomposes as
\begin{eqnarray}
& & [(\sigma^{-1} \tau - \rho_\p(\sigma^{-1} \tau)) e_\p + (\rho_\p(\sigma^{-1} \tau) - \sigma^{-1} \tau) e_{\p_0}] \nonumber \\
&+& [\rho_\p(\sigma^{-1} \tau) e_\p - \rho_\p(\sigma^{-1} \tau) e_{\p_0}] \label{conn calc eqn}
\end{eqnarray}
in $R_{S'} \oplus Z$. Also, we recognize the first bracketed element in (\ref{conn calc eqn}) as $-\sigma^{-1} r_{\sigma,\tau}^{(\p)}$. Thus
\begin{eqnarray*}
(\sigma \mu)(\tau e_\p - \tau e_{\p_0}) &=& -\sigma \mu(\sigma^{-1} r_{\sigma,\tau}^{(\p)}) \\
&=& \sigma s(\sigma^{-1} r_{\sigma,\tau}^{(\p)}) \\
&=& s(r_{\sigma,\tau}^{(\p)}) .
\end{eqnarray*}
Since $\mu(\tau e_\p - \tau e_{\p_0}) = 0$, we therefore obtain that
\[ f(\sigma)(\tau e_\p - \tau e_{\p_0}) = s(r_{\sigma,\tau}^{(\p)}) .\]
Further, the map $g : G \ra \Hom_\Z(X,\class_S(L))$ appearing in the statement of the lemma has image $f$ in $\map(G,\Hom_\Z(B_{S'},\class_S(L)))$. The lemma is proven.
\end{proof}

\begin{definition} \label{def frak c}
If $\p \in S \setm \{\p_0\}$ and $\tau \in G_\p$, then $\iota_{\p_0}(\tau)^{-1} \iota_\p(\tau) \in \gal{L_S/L}$, and we let its image in $\class_S(L)$ under the Artin map be $\frakc_\p(\tau)$.
\end{definition}

\begin{lemma} \label{snake of r}
If $\p \in S \setm \{\p_0\}$, $\tau \in D_\p$ and $\sigma \in G$, then $s(r_{\sigma,\tau}^{(\p)}) = \tau \frakc_\p(\tau^{-1} \sigma \rho_\p(\sigma^{-1} \tau))$.
\end{lemma}

\begin{proof}
Since $s$ is a $\Z[G]$-module homomorphism, the statement of the lemma is equivalent to $s(\tau^{-1} r_{\sigma,\tau}^{(\p)}) = \frakc_\p(\tau^{-1} \sigma \rho_\p(\sigma^{-1} \tau))$. Now, recalling the notation $e_\p$ introduced in the proof of Lemma \ref{image of snake},
\[ \tau^{-1} r_{\sigma,\tau}^{(\p)} = (\tau^{-1} \sigma \rho_\p(\sigma^{-1} \tau) - 1)e_\p - (\tau^{-1} \sigma \rho_\p(\sigma^{-1} \tau) - 1)e_{\p_0} ,\]
whose image in $\calH$ is
\[ \rc{\iota_\p(\tau^{-1} \sigma \rho_\p(\sigma^{-1} \tau)) - 1} - \rc{\iota_{\p_0}(\tau^{-1} \sigma \rho_\p(\sigma^{-1} \tau)) - 1} \]
by the discussion of Section \ref{sec: desc s}. However, this is simply
\begin{equation} \label{snake of r eqn}
\rc{\iota_\p(\tau^{-1} \sigma \rho_\p(\sigma^{-1} \tau)) - \iota_{\p_0}(\tau^{-1} \sigma \rho_\p(\sigma^{-1} \tau))} .
\end{equation}
Now, one checks that if $\sigma_1,\sigma_2 \in G^S$ and $\sigma_2^{-1} \sigma_1 \in \gal{L_S/L}$, then $\sigma_2^{-1} \sigma_1$ has image $\rc{\sigma_1 - \sigma_2}$ in $\calH$. Therefore the element of (\ref{snake of r eqn}) is the image of
\[ \iota_{\p_0}(\tau^{-1} \sigma \rho_\p(\sigma^{-1} \tau))^{-1} \iota_\p(\tau^{-1} \sigma \rho_\p(\sigma^{-1} \tau)) \]
under $\gal{L_S/L} \ra \calH$. Thus $s(\tau^{-1} r_{\sigma,\tau}^{(\p)}) = \frakc_\p(\tau^{-1} \sigma \rho_\p(\sigma^{-1} \tau))$ as desired.
\end{proof}

For each $\q \in S' \setm S$, let $s_\q : \Z[G] \ra \class_S(L)$ be the restriction of the snake map to the copy of $\Z[G]$ in $R_{S'}$ corresponding to $\q$.

\begin{lemma} \label{snake of 1q}
$s_\q(1) = \rc{\q(L) \oh_{L,S}}$, where $\q(L)$ is the distinguished prime of $L$ above $\q$.
\end{lemma}

\begin{proof}
The map $\Z[G] \ra R_{S'} \ra \calH$ corresponding to $\q$ sends $1$ to $\rc{\br{\frob}_\q - 1}$ (recall the definition of $\br{\frob}_\p$ in Section \ref{notation}), which is the image of $\br{\frob}_\q$ under $\gal{L_S/L} \ra \calH$. Note that $\br{\frob}_\q$ is indeed in $\gal{L_S/L}$ because $K_\q = L_\q$. We observe also that $\br{\frob}_\q$ is, in this case, just the Frobenius in $L_S/L$ associated to $\q(L)$, and therefore the image of $\rc{\q(L) \oh_{L,S}}$ under the Artin map $\class_S(L) \ra \gal{L_S/L}$ is $\br{\frob}_\q$, completing the proof.
\end{proof}

\section{The map $H^{-2}(G,X) \ra H^{-1}(G,\class_S(L))$} \label{X Cl map}

As in Section \ref{sec: ext of nabla}, we assume that $S$ contains the infinite and ramified places and that there exists a finite place $\p_0 \in S$ such that $G_{\p_0} = G$. We will assume this for the remainder of the article.

\begin{lemma} \label{gens of homology of X}
$H^{-2}(G,X)$ is generated over $\Z$ by
\[ \{\rc{[\tau] \ten (\p(L) - \p_0(L))} \sat \p \in S \setm\{\p_0\}, \tau \in G_\p\} .\]
\end{lemma}

\begin{proof}
Note that the exact sequence $0 \ra X \ra Y \ra \Z \ra 0$ gives an exact sequence
\begin{equation} \label{X hom gens seq}
0 \ra H^{-2}(G,X) \ra H^{-2}(G,Y) \ra H^{-2}(G,\Z) \ra 0 .
\end{equation}
Indeed, the assumption $G_{\p_0} = G$ implies that $H^{-1}(G,X)$ is zero and that $H^{-3}(G,Y) \ra H^{-3}(G,\Z)$ is surjective. Now, $Y$ is isomorphic to $\bigoplus_{\p \in S} \Z[G] \teno{\Z[G_\p]} \Z$, so by Lemma \ref{Hab iso}, there is an isomorphism
\begin{eqnarray*}
\bigoplus_{\p \in S} G_\p\ab &\ra& H^{-2}(G,Y) \\
(\tau_\p [G_\p,G_\p])_\p &\mapsto& \rc{\sum_{\p \in S} [\tau_\p] \ten \p(L)} .
\end{eqnarray*}

Given $\p \in S \setm \{\p_0\}$ and $\tau \in G_\p$ let $x_{\p}(\tau)$ be the element of $\bigoplus_{\p \in S} G_\p\ab$ with $\tau [G_\p,G_\p]$ in the $\p$-component, $\tau^{-1} [G_{\p_0},G_{\p_0}]$ in the $\p_0$-component and the commutator group $[G_\q,G_\q]$ in the $\q$-component for $\q \in S \setm \{\p_0,\p\}$. Observe then that $\Ker(\bigoplus_{\p \in S} G_\p\ab \ra G\ab)$ is generated by
\[ \{x_\p(\tau) \sat \p \in S \setm \{\p_0\}, \tau \in G_\p\} .\]
This makes use of the fact that $G_{\p_0} = G$. 

Now, $x_\p(\tau)$ maps to
\[ \rc{[\tau] \ten \p(L) - [\tau] \ten \p_0(L)} = \rc{[\tau] \ten (\p(L) - \p_0(L))} ,\]
which lies in the image of the injective map $H^{-2}(G,X) \ra H^{-2}(G,Y)$. Since $\Ker(\bigoplus_{\p \in S} G_\p\ab \ra G\ab)$ is generated by the elements $x_\p(\tau)$ as above, we obtain an injective map
\[ \Ker\left(\bigoplus_{\p \in S} G_\p\ab \ra G\ab\right) \ra H^{-2}(G,X) \]
which sends $x_\p(\tau)$ to $\rc{[\tau] \ten (\p(L) - \p_0(L))}$. Since both sides have the same order (they can be shown to be isomorphic by using the sequence in (\ref{X hom gens seq}) directly, together with Shapiro's Lemma), this map is an isomorphism. This completes our proof.
\end{proof}

\begin{remark}
It is clear from the definition of the elements $x_\p(\tau)$ in the above proof that, in the statement of Lemma \ref{gens of homology of X}, the elements $\tau$ need only be taken from a set of representatives for the left cosets of $[G_\p,G_\p]$ in $G_\p$.
\end{remark}

\begin{lemma} \label{cupping lemma}
Let $A$ be a $\Z[G]$-module and $\xi \in H^1(G,A)$ be represented by the $1$-cocycle $g : G \ra A$. Then the map
\[ H^{-2}(G,X) \ra H^{-1}(G,X \teno{\Z} A) \]
obtained by taking cup-product with $\xi$ sends $\rc{[\tau] \ten (\p(L) - \p_0(L))}$ to $\rc{(\p(L) - \p_0(L)) \ten g(\tau)}$, where $\p \in S \setm\{\p_0\}$ and $\tau \in G_\p$.
\end{lemma}

\begin{proof}
We dimension shift using the commutativity of the diagram
\[ \xymatrix{
H^{-2}(G,X) \ar[r]^{c_1} \ar[d]^{\delta_1} & H^{-1}(G,X \teno{\Z} A) \ar[d]^{\delta_2} \\
H^{-1}(G,\aug{G} \teno{\Z} X) \ar[r]^{c_2} & H^0(G,\aug{G} \teno{\Z} X \teno{\Z} A)
} \]
where $c_1$ and $c_2$ are obtained by cupping with $\xi$. Given $\p \in S \setm \{\p_0\}$ and $\tau \in G_\p$, $\delta_1(\rc{[\tau] \ten (\p(L) - \p_0(L))}) = \rc{(\tau - 1) \ten (\p(L) - p_0(L))}$. By \cite[Appendix to Chapter XI, Lemma 2]{serre:localfields}, the image of this under $c_2$ is represented in $H^0(G,\aug{G} \teno{\Z} X \teno{\Z} A)$ by
\begin{eqnarray*}
& & - \sum_{\sigma \in G} \sigma((\tau - 1) \ten (\p(L) - \p_0(L))) \ten g(\sigma) \\
&=& - \sum_{\sigma \in G} (\sigma \tau - 1) \ten (\sigma \p(L) - \p_0(L)) \ten g(\sigma) + \sum_{\sigma \in G} (\sigma - 1) \ten (\sigma \p(L) - \p_0(L)) \ten g(\sigma) \\
&=& - \sum_{\sigma \in G} (\sigma - 1) \ten (\sigma \tau^{-1} \p(L) - \p_0(L)) \ten g(\sigma \tau^{-1}) + \sum_{\sigma \in G} (\sigma - 1) \ten (\sigma \p(L) - \p_0(L)) \ten g(\sigma) \\
&=& - \sum_{\sigma \in G} (\sigma  - 1) \ten (\sigma \p(L) - \p_0(L)) \ten (g(\sigma \tau^{-1}) - g(\sigma)) \\
&=& - \sum_{\sigma \in G} (\sigma - 1) \ten (\sigma \p(L) - \p_0(L)) \ten \sigma g(\tau^{-1}) .
\end{eqnarray*}
The image under $\delta_2^{-1}$ of the class of this is
\begin{eqnarray*}
- \rc{(\p(L) - \p_0(L)) \ten g(\tau^{-1})} &=& \rc{(\p(L) - \p_0(L)) \ten \tau^{-1} g(\tau)} \\
&=& \rc{(\p(L) - \p_0(L)) \ten g(\tau)} .
\end{eqnarray*}
\end{proof}

\begin{prop} \label{image of conn}
Under the above assumptions, the connecting homomorphism $H^{-2}(G,X) \ra H^{-1}(G,\class_S(L))$ sends $\rc{[\tau] \ten (\p(L) - \p_0(L))}$, with $\p \in S \setm \{\p_0\}$ and $\tau \in G_\p$, to $\rc{\frakc_\p(\tau)}$. (See Definition \ref{def frak c} for the definition of $\frakc_\p(\tau)$.)
\end{prop}

\begin{proof}
We apply Lemma \ref{cupping lemma} in the case $A = \Hom_\Z(X,\class_S(L))$ and $g$ is as in Lemma \ref{image of snake}, and combine this with Lemma \ref{conn as cup}. Given $\p \in S \setm \{\p_0\}$ and $\tau \in G_\p$, the image of $\rc{[\tau] \ten (\p(L) - \p_0(L))}$ in $H^{-1}(G,\class_S(L))$ is represented by $g(\tau)(\p(L) - \p_0(L)) = s(r_{\tau,1}^{(\p)})$. By Lemma \ref{snake of r}, $s(r_{\tau,1}^{(\p)}) = \frakc_\p(\tau)$.
\end{proof}

We can now describe $H^1(G,\oh_{L,S}\st)$ as a quotient of $H^{-1}(G,\class_S(L))$.

\begin{definition} \label{calC def}
Let $\rc{\calC}_S(L/K)$ denote the subgroup of $H^{-1}(G,\class_S(L))$ generated by
\[ \{\rc{\frakc_\p(\tau)} \sat \p \in S \setm \{\p_0\}, \tau \in G_\p \setm \{1\}\} .\]
\end{definition}
(See Definition \ref{def frak c} for the definition of $\frakc_\p(\tau)$.) We remark that the bar in $\rc{\calC}_S(L/K)$ is intended to reflect that we are working in cohomology. A similar object $\calC_S(L/K)$ defined for the class-group in its entirety will be appear in Corollary \ref{kernel of norm}.

\begin{prop} \label{-1,1 iso}
There is a canonical isomorphism
\[ H^{-1}(G,\class_S(L))/\rc{\calC}_S(L/K) \can H^1(G,\oh_{L,S}\st) .\]
\end{prop}

\begin{proof}
In light of Proposition \ref{image of conn}, and the existence of the canonical isomorphism $H^{-1}(G,\nabla) \can H^1(G,\oh_{L,S}\st)$ arising from the Tate sequence (\ref{tate sequence}), it suffices to show that $H^{-1}(G,X) = 0$. To demonstrate the vanishing of this cohomology group, consider the exact sequence
\[ H^{-2}(G,Y) \ra H^{-2}(G,\Z) \ra H^{-1}(G,X) \ra 0 .\]
Viewing $H^{-2}(G,Y)$ as $\bigoplus_{\p \in S} G_\p\ab$ and $H^{-2}(G,\Z)$ as $G\ab$, we see that the map $H^{-2}(G,Y) \ra H^{-2}(G,\Z)$ is surjective, showing that indeed $H^{-1}(G,X) = 0$. We have used in particular that $G_{\p_0} = G$.
\end{proof}

\section{The map $H^{-1}(G,\class_S(L)) \ra H^1(G,\oh_{L,S}\st)$} \label{Cl U map}

We now wish to make the canonical isomorphism appearing in Proposition \ref{-1,1 iso} explicit. The Tate sequence (\ref{tate sequence}) is obtained by combining the short exact sequence
\begin{equation} \label{left tate}
0 \ra \oh_{L,S}\st \ra A \ra \Ker(s) \ra 0
\end{equation}
arising from (\ref{pretate}) with the short exact sequence
\begin{equation} \label{right tate}
0 \ra \Ker(t) \ra B_{S'} \ra \nabla \ra 0
\end{equation}
arising from the middle vertical map in (\ref{RBX diagram}), observing that $\Ker(s) = \Ker(t)$. Therefore the map $H^i(G,\class_S(L)) \ra H^{i+2}(G,\oh_{L,S}\st)$ factors as
\begin{eqnarray*}
& & H^i(G,\class_S(L)) \ra H^i(G,\nabla) \ra \\
& & H^{i+1}(G,\Ker(t)) \ra H^{i+1}(G,\Ker(s)) \stackrel{\delta_2}{\ra} H^{i+2}(G,\oh_{L,S}\st) ,
\end{eqnarray*}
where the second map is the connecting homomorphism arising from (\ref{right tate}), and the last map, $\delta_2$, is the connecting homomorphism arising from (\ref{left tate}). However, the composition of the first three maps is just the connecting homomorphism $\delta_1 : H^i(G,\class_S(L)) \ra H^{i+1}(G,\Ker(s))$ arising from the exact sequence
\[ 0 \ra \Ker(s) \ra R_{S'} \ra \class_S(L) \ra 0 .\]
In particular, the map $H^{-1}(G,\class_S(L)) \ra H^1(G,\oh_{L,S}\st)$ is the composition
\[ H^{-1}(G,\class_S(L)) \stackrel{\delta_1}{\ra} H^0(G,\Ker(s)) \stackrel{\delta_2}{\ra} H^1(G,\oh_{L,S}\st) .\]

\subsection{The first connecting homomorphism, $\delta_1$} \label{sec: delta1}

By our choice of the set $S'$, $\class_S(L)$ can be generated by the classes of the ideals of $L$ above those in $S' \setm S$. Further, since for $\frakc \in \class_S(L)$ and $\sigma \in G$ we have $\rc{\sigma \frakc} = \rc{\frakc}$ in $H^{-1}(G,\class_S(L))$, an element of this cohomology group can be represented by an ideal class of the form
\begin{equation} \label{element of class mod G}
\frakc = \sum_{\q \in S' \setm S} a_\q \rc{\q(L) \oh_{L,S}}
\end{equation}
with the $a_\q$ in $\Z$, writing $\class_S(L)$ additively. By assumption, the fractional ideal $\prod_{\q \in S' \setm S} (\q(L) \oh_{L,S})^{a_\q N}$ of $\oh_{L,S}$ is principal and therefore is generated over $\oh_{L,S}$ by $a$ for some $a \in L\st$.

By Lemma \ref{snake of 1q}, a lift of $\frakc$ to $R_{S'}$ is the element $(a_\q)_\q$ of $\bigoplus_{\q \in S' \setm S} \Z[G]$. Applying $N$, we see that $\delta_1(\rc{\frakc})$ is represented by the element $(a_\q N)_\q$ of $\Ker(s)$, that is, $a_\q N$ in the $\q$-component for $\q \in S' \setm S$ and $0$ elsewhere.

\subsection{The second connecting homomorphism, $\delta_2$} \label{sec: delta2}

We describe the map $V_{S'} \ra \mathcal{V}$. Having done this, it will be easy to see that the natural choice of a lift of $(a_\q N)_\q$ to $V_{S'}$ under $V_{S'} \ra W_{S'}$ does indeed lie in $A = \Ker(V_{S'} \ra \calV)$ -- see Lemma \ref{delta2 lift} below. The corresponding $1$-cocycle $G \ra \oh_{L,S}\st$ will then drop right out, as in Theorem \ref{image in H1 thm} below.

So, recall from (\ref{vs' def}) that $V_{S'} = (\bigoplus_{\p \in S'} \Z[G] \teno{\Z[G_\p]} V_\p) \oplus (\bigoplus_{\pp \not\in S'} U_\pp)$, where $V_\p$ is described for $\p \in S$ in \cite[Section 1]{rw:tateforglobal}. By \cite[pp.195--196]{tate:gcft}, there is $\lambda_\p \in \Hom_\Z(\aug{G_\p},C_L)$ such that we can choose the map $V_\p \ra \calV$ to send $(b,x)$ to $(\varsigma_\p(b) \lambda_\p(x),x)$ for $b \in L_\p\st$ and $x \in \aug{G_\p}$, where $\varsigma_\p : L_\p\st \ra C_L$ is the canonical inclusion. 

As for the description of $V_\q$ for $\q \in S' \setm S$, we see from \cite[Prop. 2]{rw:tateforglobal} that under our assumptions, $V_\q = U_\q \oplus \Z$ as $\Z$-modules, noting that $G_\q$ is trivial since $\q$ splits completely in $L/K$. (We are also using our notation convention that for a finite place $\q$ of $K$, $U_\q$ denotes $U_{\q(L)}$.) Thus the $\q$-component of $V_{S'}$ is $(\Z[G] \teno{\Z} U_\q) \oplus \Z[G]$. The map $(\Z[G] \teno{\Z} U_\q) \oplus \Z[G] \ra \mathcal{V}$ sends $(\alpha \ten u,\beta)$ to $(\varsigma_\q(u)^\alpha \varsigma_\q(\pi_\q)^\beta,0)$ where $\pi_\q$ is a fixed uniformizer of $L_\q$. 

Finally, all unit groups $U_\pp$ appearing in $V_{S'}$ are mapped to $0$ in $\calV$.

By definition of $a$ (see the sentence following (\ref{element of class mod G})), for all $\sigma \in G$ we have $a \sigma(\pi_\q)^{-a_\q}$ equal to a unit $\tilde{u}_{\q,\sigma}$ in $L_{\sigma \q(L)}$, which in turn can be expressed as $\sigma(u_{\q,\sigma})$ for a unique unit $u_{\q,\sigma}$ in $L_\q$. Further, $a$ is a unit in $L_\pp$ for every $\pp$ not above $S'$.

Let $v$ be the element of $V_{S'}$ which has $\sum_{\sigma \in G} \sigma \ten (\sigma^{-1}(a),0)$ in the $\p$-component for $\p \in S$, $(\sum_{\sigma \in G} \sigma \ten u_{\q,\sigma},a_\q N)$ in the $\q$-component for $\q \in S' \setm S$, and $a \in U_\pp \con L_\pp$ in the $\pp$-component for every prime $\pp$ of $L$ not above $S'$.

\begin{lemma} \label{delta2 lift}
With notation as above, the element $v$ of $V_{S'}$ is a lift of $(a_\q N)_\q$ to $A = \Ker(V_{S'} \ra \mathcal{V})$.
\end{lemma}

\begin{proof}
We first show that $v \in \Ker(V_{S'} \ra \mathcal{V})$. For $\p \in S$, the image of $\sum_{\sigma \in G} \sigma \ten (\sigma^{-1}(a),0)$ is
\[ \sum_{\sigma \in G} \sigma(\varsigma_\p(\sigma^{-1}(a)),0) = \left(\prod_{\sigma \in G} \varsigma_\p(\sigma^{-1}(a))^\sigma,0\right) ,\]
and the element of $C_L$ appearing in the left-hand entry is represented by the idele having $a$ in all the components above $\p$ and $1$ everywhere else.

Now consider $\q \in S' \setm S$. The image of $(\sum_{\sigma \in G} \sigma \ten u_{\q,\sigma},a_\q N)$ in $\calV$ is $\sum_{\sigma \in G} (\varsigma_\q(u_{\q,\sigma} \pi_\q^{a_\q})^\sigma,0)$. Using that $a = \sigma(u_{\q,\sigma} \pi_\q^{a_\q})$, the element $\prod_{\sigma \in G} \varsigma_\q(u_{\q,\sigma} \pi_\q^{a_\q})^\sigma$ of $C_L$ is represented by the idele having $a$ in all the components above $\q$ and $1$ everywhere else. Remembering the contribution of $a \in U_\pp$ for each $\pp$ not above $S'$, which gets mapped to $(x,0)$ where $x$ is the idele class represented by the idele with $a$ in the $\pp$-component and $1$ everywhere else, we find that $v$ is mapped to $(\rc{a},0) = (1,0) \in \mathcal{V}$, and so lies in $A$.

That $v$ is a lift of $(a_\q N)_\q$ is clear.
\end{proof}

Recall the ideal class $\frakc$ defined in (\ref{element of class mod G}) and the element $a$ of $L\st$ appearing just after.

\begin{theorem} \label{image in H1 thm}
The image of $\rc{\frakc}$ under $H^{-1}(G,\class_S(L)) \ra H^1(G,\oh_{L,S}\st)$ is represented by the $1$-cocycle $\tau \mapsto \tau(a)/a$.
\end{theorem}

\begin{proof}
Take $\tau \in G$. By the construction of the relevant connecting homomorphism associated to $0 \ra \oh_{L,S}\st \ra A \ra \Ker(s) \ra 0$, $\tau v - v$ must be a principal idele with $S$-unit entries, so it suffices to determine what that $S$-unit is in just one component. By considering $\pp$ not above $S'$, it is clear that it must be $\tau(a)/a$ as claimed, but we can verify this by calculating the idele in all components.

First consider $\q \in S' \setm S$.
\[ (\tau - 1)\left(\sum_{\sigma \in G} \sigma \ten u_{\q,\sigma},a_q N\right) = \left(\sum_{\sigma \in G} (\tau - 1)\sigma \ten u_{\q,\sigma},0\right) ,\]
and the element of $\Z[G] \teno{\Z} U_\q$ in the left-hand entry is
\begin{eqnarray*}
\sum_{\sigma \in G} (\tau \sigma) \ten u_{\q,\sigma} - \sum_{\sigma \in G} \sigma \ten u_{\q,\sigma} &=& \sum_{\sigma \in G} \sigma \ten (u_{\q,\tau^{-1} \sigma} u_{\q,\sigma}^{-1}) \\
&=& \sum_{\sigma \in G} \sigma \ten \sigma^{-1}(\sigma(u_{\q,\tau^{-1} \sigma}) \sigma(u_{\q,\sigma})^{-1}) .
\end{eqnarray*}
Under the identification of $\Z[G] \teno{\Z} U_\q$ with $\bigoplus_{\qq|\q} U_\qq$, this corresponds to the element of $\bigoplus_{\qq|\q} U_\qq$ having $\sigma(u_{\q,\tau^{-1} \sigma}) \sigma(u_{\q,\sigma})^{-1}$ in the $\sigma \q(L)$-component. However, by definition $\sigma(u_{\q,\tau^{-1} \sigma}) \sigma(u_{\q,\sigma})^{-1} = \tau(a)/a$.

Now consider $\p \in S$.
\begin{eqnarray*}
(\tau - 1) \sum_{\sigma \in G} \sigma \ten (\sigma^{-1}(a),0) &=& \sum_{\sigma \in G} (\tau \sigma - \sigma) \ten (\sigma^{-1}(a),0) \\
&=& \sum_{\sigma \in G} \sigma \ten (\sigma^{-1}(\tau(a)/a),0) ,
\end{eqnarray*}
which is the image under $\bigoplus_{\pp|\p} L_\pp\st \ra V_{S'}$ of the element having $\tau(a)/a$ in every component. This completes the verification.
\end{proof}

Incidentally, as an extra check, we can verify independently that for each $\tau \in G$, $\tau(a)/a \in \oh_{L,S}\st$. Namely, since
\[ a \oh_{L,S} = \prod_{\q \in S' \setm S} (\q(L) \oh_{L,S})^{a_\q N} \]
and $\tau N = N$, we see that $\tau(a) \oh_{L,S} = a \oh_{L,S}$, \ie $\tau(a)/a \in \oh_{L,S}\st$.

\section{Corollaries} \label{cors}

We give a sequence of corollaries of Theorem \ref{image in H1 thm}, which lead to Corollary \ref{kernel of norm}.

Recall once again the ideal class $\frakc$ defined in (\ref{element of class mod G}) and the element $a$ of $L\st$ defined just after.

\begin{cor} \label{thm cor 1}
With the above notation, $\rc{\frakc} \in \Ker(H^{-1}(G,\class_S(L)) \ra H^1(G,\oh_{L,S}\st))$ if and only if $a \in \oh_{L,S}\st \cdot K\st$.
\end{cor}

\begin{proof}
By Theorem \ref{image in H1 thm}, the cohomology class $\rc{\frakc}$ lies in the specified kernel if and only if the $1$-cocycle
\begin{eqnarray*}
G &\ra& \oh_{L,S}\st \\
\sigma &\mapsto& \sigma(a)/a
\end{eqnarray*}
is a $1$-coboundary. But $\sigma \mapsto \sigma(a)/a$ is a $1$-coboundary if and only if there is $u \in \oh_{L,S}\st$ such that $\sigma(a)/a = \sigma(u)/u$ for all $\sigma \in G$, which is the same as saying that $a/u$ is fixed by $G$, \ie $a/u \in K\st$.
\end{proof}

Let $\calF$ be the kernel of the composition
\[ \bigoplus_{\q \in S' \setm S} \Z \ra \class_S(L) \stackrel{N}{\ra} \class_S(L) ,\]
where the first map is given by sending an element $(a_\q)_\q$ to the class of the ideal $\prod_{\q \in S' \setm S} (\q(L) \oh_{L,S})^{a_\q}$. (There is no relationship with the $\calF$ used in the proof of Lemma \ref{Hab iso}.) Thus $\calF$ surjects onto $H^{-1}(G,\class_S(L))$. Note that we also have a map $\calF \ra \class_S(K)$ given by restricting
\begin{eqnarray*}
\bigoplus_{\q \in S' \setm S} \Z &\ra& \class_S(K) \\
(a_\q)_\q &\mapsto& \rc{\prod_{\q \in S} (\q \oh_{K,S})^{a_\q}}
\end{eqnarray*}
to $\calF$.

\begin{cor} \label{calF same kernels}
The maps
\[ \calF \ra H^{-1}(G,\class_S(L)) \ra H^1(G,\oh_{L,S}\st) \]
and
\[ \calF \ra \class_S(K) \]
have the same kernel.
\end{cor}

\begin{proof}
Take $(a_\q)_\q \in \calF$ in the kernel of the first map, that is to say
\[ \prod_{\q \in S' \setm S} (\q(L) \oh_{L,S})^{a_\q N} = a \oh_{L,S} \]
for some $a \in \oh_{L,S}\st \cdot K\st$ by Corollary \ref{thm cor 1}. Of course, we may assume $a \in K\st$. Since $\q$ splits completely in $L/K$ for $\q \in S' \setm S$,
\[ \nu_{\q}(a) = \nu_{\q(L)}(a) = a_\q .\]
Since $a \in \oh_{K,S'}\st$, we have
\[ a \oh_{K,S} = \prod_{\q \in S' \setm S} (\q \oh_{K,S})^{a_\q} ,\]
and so the class of $\prod_{\q \in S' \setm S} (\q \oh_{K,S})^{a_\q}$ is zero in $\class_S(K)$. Thus $(a_\q)_\q$ is in the kernel of the second map.

Conversely, take $(a_\q)_\q \in \calF$ such that $\prod_{\q \in S' \setm S} (\q \oh_{K,S})^{a_\q}$ is principal as an ideal of $\oh_{K,S}$, \ie
\[ \prod_{\q \in S' \setm S} (\q \oh_{K,S})^{a_\q} = b \oh_{K,S} \]
for some $b \in K\st$. Choose $a \in L\st$ such that $\prod_{\q \in S' \setm S} (\q(L) \oh_{L,S})^{a_\q N} = a \oh_{L,S}$, as in the construction of the map $H^{-1}(G,\class_S(L)) \ra H^1(G,\oh_{L,S}\st)$. Then
\begin{eqnarray*}
a \oh_{L,S} &=& \prod_{\q \in S' \setm S} (\q(L) \oh_{L,S})^{a_\q N} \\
&=& \left(\prod_{\q \in S' \setm S} \q^{a_\q}\right) \oh_{L,S} \\
&=& b \oh_{L,S} ,
\end{eqnarray*}
so $a \in \oh_{L,S}\st \cdot K\st$. By Corollary \ref{thm cor 1}, this says that $(a_\q)_\q$ is in the kernel of the first map.
\end{proof}

\begin{cor} \label{units and class embedding}
There is an embedding $H^1(G,\oh_{L,S}\st) \ra \class_S(K)$.
\end{cor}

\begin{proof}
By Corollary \ref{calF same kernels}, we have
\[ H^1(G,\oh_{L,S}\st) \stackrel{\can}{\ra} \calF/\Ker(\calF \ra \class_S(K)) \ra \class_S(K) ,\]
the second map being injective.
\end{proof}

We observe that the above corollary gives an alternative justification that $H^1(G,\oh_{L,S}\st)$ embeds into $\class_S(K)$ to that found in \cite[Cor. 2]{rim:exactseq}.

\begin{definition}
Let $\inorm : \class_S(L) \ra \class_S(K)$ be the homomorphism on class-groups induced by the norm of ideals. Denote the resulting map $H^{-1}(G,\class_S(L)) \ra \class_S(K)$ by $\rc{\inorm}$.
\end{definition}

\begin{cor} \label{kernel of Nm bar}
There is an exact sequence
\[ 0 \ra \rc{\calC}_S(L/K) \ra H^{-1}(G,\class_S(L)) \stackrel{\rc{\inorm}}{\ra} \class_S(K) .\]
(Recall the definition of $\rc{\calC}_S(L/K)$ in Definition \ref{calC def}.)
\end{cor}

\begin{proof}
By Proposition \ref{-1,1 iso}, $\rc{\calC}_S(L/K)$ is the kernel of $H^{-1}(G,\class_S(L)) \ra H^1(G,\oh_{L,S}\st)$, which is the kernel of
\begin{equation} \label{should be norm}
H^{-1}(G,\class_S(L)) \ra H^1(G,\oh_{L,S}\st) \ra \class_S(K)
\end{equation}
by Corollary \ref{units and class embedding}. It remains to show that the map in (\ref{should be norm}) is $\rc{\inorm}$. An element $x$ of $H^{-1}(G,\class_S(L))$ is represented by an ideal class containing an ideal of the form
\[ \prod_{\q \in S' \setm S} (\q(L) \oh_{L,S})^{a_\q} \]
with $(a_\q)_\q \in \calF$. If $y$ is the image of $x$ in $H^1(G,\oh_{L,S}\st)$, then the image of $y$ in $\class_S(K)$ is obtained by first lifting $y$ to $\calF$ via $H^{-1}(G,\class_S(L))$, and then applying the map $\calF \ra \class_S(K)$. However, such a lift of $y$ is $(a_\q)_\q$, whose image in $\class_S(K)$ is the ideal class containing
\[ \prod_{\q \in S' \setm S} (\q \oh_{K,S})^{a_\q} .\]
This ideal class is clearly $\rc{\inorm}(x)$.
\end{proof}

Corollary \ref{kernel of Nm bar} establishes that $\rc{\calC}_S(L/K) = \Ker(\rc{\inorm})$. We remark in passing, although we will not use this fact, that this kernel is in fact the ``divisor knot'' as defined by Jehne in \cite[Ch.I, Section 1]{jehne:knots} -- see the isomorphism before Theorem 1 of \cite{jehne:knots} and the definition in \cite[(1.5)]{jehne:knots}. (Jehne's divisor knot contains the earlier divisor knot of Scholz \cite{scholz:normenreste}.)

Now, consider the following definition:

\begin{definition} \label{def D}
Let $D(\class_S(L))$ be the subgroup of $\class_S(L)$ generated by
\[ \{(\sigma - 1) \frakc \sat \frakc \in \class_S(L), \sigma \in G\} .\]
\end{definition}

Observe that $D(\class_S(L))$ is closed under the action of $G$ and is therefore a $\Z[G]$-submodule of $\class_S(L)$.

Noting that
\begin{equation} \label{kernel of Nm inclusions}
D(\class_S(L)) \con \Ker(\inorm) \con \Ker(N : \class_S(L) \ra \class_S(L)) ,
\end{equation}
we obtain an exact sequence
\[ 0 \ra D(\class_S(L)) \ra \Ker(\inorm) \ra \Ker(\rc{\inorm}) \ra 0 .\]
Thus we arrive at an exact sequence
\begin{equation} \label{break up kernel of Nm}
0 \ra D(\class_S(L)) \ra \Ker(\inorm) \ra \rc{\calC}_S(L/K) \ra 0 .
\end{equation}

\begin{definition} \label{def calC}
Let $\calC_S(L/K)$ be the subgroup of $\class_S(L)$ generated by
\[ \{\frakc_\p(\tau) \sat \p \in S \setm \{\p_0\}, \tau \in G_\p \setm \{1\}\} .\]
\end{definition}

The following gives a description of the kernel of the norm map $\class_S(L) \ra \class_S(K)$.

\begin{cor} \label{kernel of norm}
There is an exact sequence
\[ 0 \ra D(\class_S(L)) + \calC_S(L/K) \ra \class_S(L) \stackrel{\inorm}{\ra} \class_S(K) \ra 0 .\]
(But see the remark, after the proof, concerning existing literature.)
\end{cor}

\begin{proof}
We see from (\ref{kernel of Nm inclusions}) and Corollary \ref{kernel of Nm bar} that $D(\class_S(L))$ and $\calC_S(L/K)$ are in the kernel of $\inorm$. Conversely, take $\frakc \in \Ker(\inorm)$, so that $\rc{\frakc} \in \Ker(\rc{\inorm})$, and write $\rc{\frakc} \in H^{-1}(G,\class_S(L))$ as
\[ \rc{\frakc} = \sum_{\p \in S \setm \{\p_0\}} \sum_{\tau \in G_\p \setm \{1\}} a_{\p,\tau} \rc{\frakc_\p(\tau)} \]
with $a_{\p,\tau} \in \Z$. Then $\frakc - \sum_{\p \in S \setm \{\p_0\}} \sum_{\tau \in G_\p \setm \{1\}} a_{\p,\tau} \frakc_\p(\tau)$ is in the kernel of $\Ker(\inorm) \ra \rc{\calC}_S(L/K)$ and therefore is in $D(\class_S(L))$ by (\ref{break up kernel of Nm}), thus showing that $\frakc \in D(\class_S(L)) + \calC_S(L/K)$.

As for the surjectivity of $\class_S(L) \ra \class_S(K)$, this follows from the fact that $\p_0$ remains non-split in $L/K$.
\end{proof}

\begin{remark}
The author is aware that a description of the kernel of the norm map as found in Corollary \ref{kernel of norm} should follow from work of, for example, Jaulent \cite{jaulent:these}, Jehne \cite{jehne:knots}, Fr\"ohlich \cite{frohlich:central} and Furuta \cite{furuta:genus} on knots, central classes and genus fields. We include our derivation of this description as an illustration of a different approach, via Ritter and Weiss's version of the Tate sequence, a derivation that, after our having obtained Theorem \ref{image in H1 thm}, is reasonably short.
\end{remark}

\subsection*{Acknowledgments}

The author would like to thank Al Weiss for a number of interesting discussions, and the referee for providing helpful remarks.

\end{document}